\documentclass{article}%
\usepackage{amsfonts}
\usepackage{amsthm}
\usepackage{amsmath}
\usepackage{amssymb}
\usepackage{graphicx}%
\providecommand{\U}[1]{\protect\rule{.1in}{.1in}}
\newtheorem{theorem}{Theorem}

\newtheorem{definition}[theorem]{Definition}

\newtheorem{lemma}[theorem]{Lemma}

\begin{document}

\title{Analytical solution of the fractional linear time-delay systems and their
Ulam-Hyers stability}
\author{Nazim I. Mahmudov\\Eastern Mediterranean University \\Department of Mathematics \\Famagusta, 99628 T. R. Northen Cyprus\\Mersin 10 Turkey}
\date{20 March 2021}
\maketitle

\begin{abstract}
\bigskip We introduce the delayed Mittag-Leffler type matrix functions,
delayed fractional cosine, delayed fractional sine and use the Laplace
transform to obtain an analytical solution to the IVP for a Hilfer type
fractional linear time-delay system $D_{0,t}^{\mu,\nu}z\left(  t\right)
+Az\left(  t\right)  +\Omega z\left(  t-h\right)  =f\left(  t\right)  $ of
order $1<\mu<2$ and type $0\leq\nu\leq1,$ with nonpermutable matrices $A$ and
$\Omega$. Moreover, we study Ulam-Hyers stability of the Hilfer type
fractional linear time-delay system. Obtained results extend those for Caputo
and Riemann-Liouville type fractional linear time-delay systems and new even
for these fractional delay systems.

\end{abstract}

\section{Introduction and auxiliary lemmas}

Khusainov et al. \cite{12} studied the following Cauchy problem for a second
order linear differential equation with pure delay:%
\begin{equation}
\left\{
\begin{array}
[c]{c}%
x^{\prime\prime}\left(  t\right)  +\Omega^{2}x\left(  t-\tau\right)  =f\left(
t\right)  ,\ \ t\geq0,\ \tau>0,\\
x\left(  t\right)  =\varphi\left(  t\right)  ,\ \ x^{\prime}\left(  t\right)
=\varphi^{\prime}\left(  t\right)  ,\ \ -\tau\leq t\leq0,
\end{array}
\right.  \label{os1}%
\end{equation}
where $f:\left[  0,\infty\right)  \rightarrow\mathbb{R}^{n}$, $\Omega$ is a
$n\times n$ nonsingular matrix, $\tau$ is the time delay and $\varphi$ is an
arbitrary twice continuously differentiable vector function. A solution of
(\ref{os1}) has an explicit representation of the form \cite[Theorem 2]{12}:%
\begin{align*}
x\left(  t\right)   &  =\left(  \cos_{\tau}\Omega t\right)  \varphi\left(
-\tau\right)  +\Omega^{-1}\left(  \sin_{\tau}\Omega t\right)  \varphi^{\prime
}\left(  -\tau\right) \\
&  +\Omega^{-1}\int_{-\tau}^{0}\sin_{\tau}\Omega\left(  t-\tau-s\right)
\varphi^{\prime\prime}\left(  s\right)  ds\\
&  +\Omega^{-1}\int_{0}^{t}\sin_{\tau}\Omega\left(  t-\tau-s\right)  f\left(
s\right)  ds,
\end{align*}
where $\cos_{\tau}\Omega:\mathbb{R}\rightarrow\mathbb{R}^{n\times n}$ and
$\sin_{\tau}\Omega:\mathbb{R}\rightarrow\mathbb{R}^{n\times n}$ denote the
delayed matrix cosine of polynomial degree $2k$ on the intervals $\left(
k-1\right)  \tau\leq t<k\tau$ and the delayed matrix sine of polynomial degree
$2k+1$ on the intervals $\left(  k-1\right)  \tau\leq t<k\tau$, respectively.

It should be stressed out that pioneer works \cite{12}, \cite{13} led to many
new results in integer and noninteger order time-delay differential equations
and discrete delayed system; see \cite{1}-\cite{11}, \cite{14}, \cite{15},
\cite{18}-\cite{21}.

Introducing the fractional analogue delayed matrices cosine/sine of a
polynomial degree, see formulas (\ref{cos1}) and (\ref{sin1}), Liang et al.
\cite{16} gave representation of a solution to the initial value problem
(\ref{w1}):

\begin{theorem}
\label{thm:wang1}\cite{16} Let $h>0$, $\varphi\in C^{2}\left(  \left[
-h,0\right]  ,R^{n}\right)  $, $\Omega$ be a nonsingular $n\times n$ matrix,
and let $f:\left[  0,\infty\right)  \rightarrow R^{n}$ be a given function.
The solution $x:\left[  -h,\infty\right)  \rightarrow R^{n}$ of the intial
value problem
\begin{equation}
\left\{
\begin{array}
[c]{c}%
^{C}D_{-h}^{\alpha}\left(  ^{C}D_{-h}^{\alpha}\right)  x\left(  t\right)
+\Omega^{2}x\left(  t-h\right)  =0,\ t\geq0,\ h>0,\\
x\left(  t\right)  =\varphi\left(  t\right)  ,\ \ x^{\prime}\left(  t\right)
=\varphi^{\prime}\left(  t\right)  ,\ \ -h\leq t\leq0,
\end{array}
\right.  \label{w1}%
\end{equation}
has the form
\begin{align*}
x\left(  t\right)   &  =\left(  \cos_{h,\alpha}\Omega t^{\alpha}\right)
\varphi\left(  -h\right)  +\Omega^{-1}\left(  \sin_{h,\alpha}\Omega\left(
t-h\right)  ^{\alpha}\right)  \varphi^{\prime}\left(  0\right)  \\
&  +\Omega^{-1}\int_{-h}^{0}\cos_{h,\alpha}\Omega\left(  t-h-s\right)
^{\alpha}\varphi^{\prime}\left(  s\right)  ds
\end{align*}
where $\cos_{h,\alpha}\Omega t^{\alpha}$ is the fractional delayed matrix
cosine of a polynomial of degree $2k\alpha$ on the intervals $\left(
k-1\right)  h\leq t<kh$, $\sin_{h,\alpha}\Omega t^{\alpha}$ is the fractional
delayed matrix sine of a polynomial of degree $\left(  2k+1\right)  \alpha$ on
the intervals $\left(  k-1\right)  h\leq t<kh$ defined as follows
\begin{equation}
\cos_{h,\alpha}\Omega t^{\alpha}:=\left\{
\begin{tabular}
[c]{ll}%
$\Theta,$ & $-\infty<t<-h,$\\
$I,$ & $-h\leq t<0,$\\
& \\
$I-\Omega^{2}\dfrac{t^{2\alpha}}{\Gamma\left(  2\alpha+1\right)  }+...+\left(
-1\right)  ^{k}\Omega^{2k}\dfrac{\left(  t-\left(  k-1\right)  h\right)
^{2k\alpha}}{\Gamma\left(  2k\alpha+1\right)  },$ & $\left(  k-1\right)  h\leq
t<kh,$%
\end{tabular}
\ \ \right.  \label{cos1}%
\end{equation}%
\begin{equation}
\sin_{h,\alpha}\Omega t^{\alpha}:=\left\{
\begin{tabular}
[c]{ll}%
$\Theta,$ & $-\infty<t<-h,$\\
$\Omega\dfrac{\left(  t+h\right)  ^{\alpha}}{\Gamma\left(  \alpha+1\right)  }$
$,$ & $-h\leq t<0,$\\
& \\
$\Omega\dfrac{\left(  t+h\right)  ^{\alpha}}{\Gamma\left(  \alpha+1\right)
}-\Omega^{3}\dfrac{t^{3\alpha}}{\Gamma\left(  3\alpha+1\right)  }$ & \\
$+...+\left(  -1\right)  ^{k}\Omega^{2k+1}\dfrac{\left(  t-\left(  k-1\right)
h\right)  ^{\left(  2k+1\right)  \alpha}}{\Gamma\left(  \left(  2k+1\right)
\alpha+1\right)  },$ & $\left(  k-1\right)  h\leq t<kh,$%
\end{tabular}
\ \ \ \right.  \label{sin1}%
\end{equation}
and $I$ is the identity matrix and $\Theta$ is the null matrix.
\end{theorem}

Mahmudov in \cite{17} studied the following R-L linear fractional differential
delay equation of order $1<2\alpha\leq2$ by introducing the concept of
fractional delayed matrix cosine $\cos_{h,\alpha,\beta}\left\{  A,\Omega
;t\right\}  $ and sine $\sin_{h,\alpha,\beta}\left\{  A,\Omega;t\right\}  $
\cite[Definitions 2 and 3]{17}.%

\begin{equation}
\left\{
\begin{array}
[c]{c}%
^{RL}D_{-h+}^{\alpha}\left(  ^{RL}D_{-h+}^{\alpha}\right)  x\left(  t\right)
+A^{2}x\left(  t\right)  +\Omega^{2}x\left(  t-h\right)  =f\left(  t\right)
,\ t\geq0,\ h>0,\\
x\left(  t\right)  =\varphi\left(  t\right)  ,\ \ \ \left(  I_{-h^{+}%
}^{1-\alpha}x\right)  \left(  -h^{+}\right)  =\varphi\left(  -h\right)
,\ -h\leq t\leq0,\\
\ ^{RL}D_{-h+}^{\alpha}x\left(  t\right)  =\ ^{RL}D_{-h+}^{\alpha}%
\varphi\left(  t\right)  ,\ \ \\
\left(  I_{-h^{+}}^{1-\alpha}\left(  D_{-h+}^{\alpha}x\right)  \right)
\left(  -h^{+}\right)  =\ \ ^{RL}D_{-h+}^{\alpha}\varphi\left(  -h\right)
,\ -h\leq t\leq0,
\end{array}
\right.  \label{pr1}%
\end{equation}
where $^{RL}D_{-h+}^{\alpha}$ stands for the R-L fractional derivative of
order $0<\alpha\leq1$ with lower limit $-h$, $A,\Omega\in\mathbb{R}^{n\times
n}$, $f\in C\left(  \left[  0,\infty\right)  ,\mathbb{R}^{n}\right)  $,
$\varphi\in C^{1}\left(  \left[  -h,0\right]  ,\mathbb{R}^{n}\right)  $.
Obviously, the derivative can be started at $-h$ instead of $0$, since the
function $x\left(  t\right)  $ governed by (\ref{pr1}) actually originates at
$-h.$ However, as is known, changing the starting point of the derivative
modifies the derivative and leads to a different problem. In this article, we
study the case when the derivative started at $0$.

We study the following Hilfer type linear fractional differential time-delay
equation of order $1<\mu<2$ and type $0\leq\nu\leq1$:
\begin{equation}
\left\{
\begin{array}
[c]{c}%
D_{0,t}^{\mu,\nu}z\left(  t\right)  +Az\left(  t\right)  +\Omega z\left(
t-h\right)  =f\left(  t\right)  ,\ \ \ t\in\left[  0,T\right]  ,\\
z\left(  t\right)  =\varphi\left(  t\right)  ,\ \ -h\leq t\leq0,\\
\left.  D_{0,t}^{-\left(  2-\mu\right)  \left(  1-\nu\right)  +1,\nu}z\left(
t\right)  \right\vert _{t=0}=a_{1},\\
\left.  D_{0,t}^{-\left(  2-\mu\right)  \left(  1-\nu\right)  ,\nu}z\left(
t\right)  \right\vert _{t=0}=a_{0},
\end{array}
\right.  \label{eq1}%
\end{equation}
where $D_{0,t}^{\mu,\nu}$ stands for the Hilfer fractional derivative of order
$1<\mu<2$ and type $0\leq\nu\leq1$ with lower limit $0$, $A,\Omega
\in\mathbb{R}^{d\times d}$, $f\in C\left(  \left[  0,T\right]  ,\mathbb{R}%
^{d}\right)  $, $\varphi\in C^{1}\left(  \left[  -h,0\right]  ,\mathbb{R}%
^{d}\right)  .$

Delayed perturbation of Mittag-Leffler matrix functions serves as a suitable
tool for solving linear fractional continuous time-delay equations. First, the
delayed matrix exponential function (delayed matrix Mittag-Leffler function)
was defined to solve linear purely delayed (fractional) systems of order one.
Then, the second order differential systems with pure delay were considered
and suitable delayed sine and cosine matrix functions were introduced in
\cite{12}. Later Liang et al. \cite{16} introduced the fractional analogue of
delayed cosine/sine matrices and obtained an explicit solution of the
sequential fractional Caputo type equations with pure delay, the case
$A=\Theta$ (zero matrix). Recently, Mahmudov \cite{17} introduced the
fractional analogue of delayed matrices cosine/sine in the case when $A$ and
$\Omega$ commutes to solve the sequential Riemann-Liouville type linear
time-delay system. It should be noticed that the model investigated here is
not sequential and differs from that of discussed in \cite{16}, \cite{17}. For
the sake of completeness, we also refer to studies of discrete/continuous
variants of delayed matrices used to obtain exact solution to linear
difference equations with delays \cite{1}-\cite{21}.

We introduce a concept of delayed Mittag-Leffler type matrix function of two parameters:

\begin{definition}
\label{def:sine}Delayed Mittag-Leffler type matrix function of two parameters
$Y_{\mu,\gamma}^{h}:\left[  0,\infty\right)  \rightarrow\mathbb{R}^{d}$ is
defined as follows:
\[
Y_{\mu,\gamma}^{h}\left(  A,\Omega;t\right)  =Y_{\mu,\gamma}^{h}\left(
t\right)  :=%
{\displaystyle\sum_{m=0}^{\infty}}
{\displaystyle\sum\limits_{k=0}^{\infty}}
\left(  -1\right)  ^{k}Q_{k,m}^{A,\Omega}\dfrac{\left(  t-mh\right)
_{+}^{k\mu+\gamma-1}}{\Gamma\left(  k\mu+\gamma\right)  },
\]
where $\left(  t\right)  _{+}=\max\left\{  0,t\right\}  $ and
\begin{align}
Q_{k,m}  &  =Q_{k,m}^{A,\Omega}=%
{\displaystyle\sum\limits_{j=m}^{k}}
A^{k-j}\Omega Q_{j-1,m-1}^{A,\Omega},\ \ .\label{qq1}\\
Q_{0,m}^{A,\Omega}  &  =Q_{k,-1}^{A,\Omega}=\Theta,\ \ Q_{k,0}^{A,\Omega
}=A^{k}\ \ Q_{0,0}^{A,\Omega}=I,\ \ k=0,1,2,...,m=0,1,2,...\nonumber
\end{align}

\end{definition}

In this definition $Q_{k,m}^{A,\Omega}$ plays the role of a kernel, see
\cite{17}, \cite{20}. It is clear that%
\begin{align*}
Q_{k+1,m}^{A,\Omega} &  =AQ_{k,m}^{A,\Omega}+\Omega Q_{k,m-1}^{A,\Omega},\\
Q_{0,m}^{A,\Omega} &  =Q_{k,-1}^{A,\Omega}=\Theta,\ \ Q_{0,0}^{A,\Omega}=I,\\
k &  =0,1,2,...,m=0,1,2,...
\end{align*}

We state the main novelties of our article as below:

\begin{itemize}
\item We introduce a novel delayed Mittag-Leffler type matrix function
$Y_{\mu,\gamma}^{h}\left(  A,\Omega;t\right)  $.

\item If $\Omega=\Theta,$ then%
\begin{align*}
Y_{2,1}^{h}\left(  A^{2},\Theta;t\right)   &  =%
{\displaystyle\sum\limits_{k=0}^{\infty}}
\left(  -1\right)  ^{k}A^{2k}\dfrac{t^{2k}}{\left(  2k\right)  !}=\cos\left(
At\right)  ,\ \ \ \\
AY_{2,2}^{h}\left(  A^{2},\Theta;t\right)   &  =A%
{\displaystyle\sum\limits_{k=0}^{\infty}}
\left(  -1\right)  ^{k}A^{2k}\dfrac{t^{2k+1}}{\left(  2k+1\right)  !}%
=\sin\left(  At\right)  ,\\
Y_{\mu,1}^{h}\left(  A^{2},\Theta;t\right)   &  =%
{\displaystyle\sum\limits_{k=0}^{\infty}}
\left(  -1\right)  ^{k}A^{2k}\dfrac{t^{k\mu}}{\Gamma\left(  k\mu+1\right)
},\ \ \ \\
Y_{\mu,2}^{h}\left(  A^{2},\Theta;t\right)   &  =%
{\displaystyle\sum\limits_{k=0}^{\infty}}
\left(  -1\right)  ^{k}A^{2k}\dfrac{t^{k\mu+1}}{\Gamma\left(  k\mu+2\right)
},\ 1<\mu<2.
\end{align*}
$Y_{\mu,1}^{h}\left(  A^{2},\Theta;t\right)  $ and $Y_{\mu,2}^{h}\left(
A^{2},\Theta;t\right)  $ can be called fractional cosine and sine for
$1<\mu<2.$ Similar cosine/sine matrix functions were defined in \cite{16},
\cite{19} to solve $1<2\mu<2$ order sequential fractional differential equations.

\item If $A=\Theta$ then we have
\[
Q_{m,m}^{A,\Omega}=\Omega^{m},\ \ \ \ Y_{\mu,\gamma}^{h}\left(  A,\Omega
;t\right)  =%
{\displaystyle\sum_{m=0}^{\infty}}
\left(  -1\right)  ^{m}\Omega^{m}\dfrac{\left(  t-mh\right)  _{+}^{m\mu
+\gamma-1}}{\Gamma\left(  m\mu+\gamma\right)  }.
\]
Moreover,%
\begin{align*}
Y_{\mu,1}^{h}\left(  \Theta,\Omega^{2};t\right)   &  =%
{\displaystyle\sum_{m=0}^{\infty}}
\left(  -1\right)  ^{m}\Omega^{2m}\dfrac{\left(  t-mh\right)  _{+}^{m\mu}%
}{\Gamma\left(  m\mu+1\right)  }=\cos_{\mu}^{h}\left(  \Omega;t\right)  ,\ \\
\Omega Y_{\mu,2}^{h}\left(  \Theta,\Omega^{2};t\right)   &  =\Omega%
{\displaystyle\sum_{m=0}^{\infty}}
\left(  -1\right)  ^{m}\Omega^{2m}\dfrac{\left(  t-mh\right)  _{+}^{m\mu}%
}{\Gamma\left(  m\mu+2\right)  }=\Omega\sin_{\mu}^{h}\left(  \Omega;t\right)
.
\end{align*}
Similar delayed cosine/sine matrix functions were defined in \cite{2},
\cite{19} to solve $1<2\mu<2$ order sequential fractional linear differential
equations with pure delay.

\item We give an exact analytical solution of the Hilfer type fractional
problem (\ref{eq1}) using delayed Mittag-Leffler type matrix function
$Y_{\mu,\gamma}^{h}\left(  A,\Omega;t\right)  $ and study their Ulam-Hyers
stability. Obtained results are new even for Caputo and Riemann-Liouville type
fractional linear time-delay systems.

\item Although the problem considered by us is fractional of order $1<\mu<2$
and type $0\leq\nu\leq1$, our approach is also applicable to the classical
second-order equations. Thus our results are new even for the classical second
order oscillatory system.
\end{itemize}

Before introducing properties of $Y_{\mu,\gamma}^{h}\left(  A,\Omega;t\right)
,$ we recall the definition of the Hilfer fractional derivative and Ulam-Hyers stability:

\begin{definition}
Let $m\in\mathbb{N},$ $m-1<\mu<m$, $0\leq\nu\leq1$, $a\in\mathbb{R}$, and
$f\in C^{m}\left[  a,b\right]  $. Then the Hilfer fractional derivative of $f$
of order $\mu$ and type $\nu$ is given by%
\[
D_{a,t}^{\mu,\nu}f\left(  t\right)  :=I_{a,t}^{\nu\left(  m-\mu\right)  }%
\frac{d^{m}}{dt^{m}}I_{a,t}^{\left(  1-\nu\right)  \left(  m-\mu\right)
}f\left(  t\right)  ,
\]
where%
\[
I_{a,t}^{\gamma}f\left(  t\right)  :=\frac{1}{\Gamma\left(  \gamma\right)
}\int_{a}^{t}\left(  t-s\right)  ^{\gamma-1}f\left(  s\right)  ds
\]
is the R-L fractional integral of $f$ of order $\gamma>0$.
\end{definition}

The main tool we use in this paper is the Laplace transform $F\left(
s\right)  :=L\left\{  f\left(  t\right)  \right\}  =\int_{0}^{\infty}%
e^{-st}f\left(  t\right)  dt,$ $\operatorname{Re}s>a$, which is defined for an
exponentially bounded function $f.$ Here are some of properties of the Laplace transform.

\begin{lemma}
\label{lem:1}The following equalities hold true for sufficiently large
\emph{Re}$\left(  s\right)  $ and appropriate functions $f,g$:

\begin{enumerate}
\item[(i)] $L\left\{  af\left(  t\right)  +bg\left(  t\right)  \right\}
=aL\left\{  f\left(  t\right)  \right\}  +bL\left\{  g\left(  t\right)
\right\}  ,\ \ \ a,b\in\mathbb{R};$

\item[(ii)] $L^{-1}\left\{  e^{-sh}s^{-1}\right\}  =1,\ \ t\geq h\geq0;$

\item[(iii)] $L^{-1}\left\{  F\left(  s\right)  G\left(  s\right)  \right\}
=\left(  f\ast g\right)  \left(  t\right)  ;$

\item[(iv)] $L\left\{  D_{0,t}^{\mu,\nu}f\left(  t\right)  \right\}  =s^{\mu
}L\left\{  f\left(  t\right)  \right\}  -%
{\displaystyle\sum\limits_{k=0}^{m-1}}
s^{m\left(  1-\nu\right)  +\mu\nu-k-1}I_{0,t}^{\left(  1-\nu\right)  \left(
m-\mu\right)  -k}f\left(  0\right)  ;$

\item[(v)] $L^{-1}\left\{  1\right\}  =\delta\left(  t\right)  ,\ $where
$\delta\left(  t\right)  $ is Dirac delta distribution;

\item[(vi)] $L^{-1}\left\{  e^{-nsh}s^{-n}\right\}  =\frac{\left(
t-nh\right)  _{+}^{n-1}}{\left(  n-1\right)  !},\ \ h>0$, $n\in\mathbb{N}$;

\item[(vii)] $L^{-1}\left\{  e^{-sh}F\left(  s\right)  \right\}  =f\left(
t-h\right)  ,\ \ h\geq0$;

\item[(viii)] $L^{-1}\left\{  e^{-sh}s^{\alpha\gamma-\beta}\left(  s^{\alpha
}I-A\right)  ^{-\gamma}\right\}  =\left(  t-h\right)  ^{\beta-1}%
E_{\alpha,\beta}^{\gamma}\left(  A\left(  t-h\right)  ^{\alpha}\right)  $,
$t\geq h$, where $E_{\alpha,\beta}^{\gamma}\left(  z\right)  =%
{\displaystyle\sum\limits_{k=0}^{\infty}}
\dfrac{z^{\alpha k}}{\Gamma\left(  \alpha k+\beta\right)  }\dfrac{\left(
\gamma\right)  _{k}}{k!}$ is the three parameter Mittag-Leffler function,
$\alpha,\beta,\gamma>0$, $t\in\mathbb{R}$ and $\left(  \gamma\right)
_{k}:=\gamma\left(  \gamma+1\right)  ...\left(  \gamma+k-1\right)  .$
\end{enumerate}
\end{lemma}

\begin{definition}
\label{def:2}System (\ref{eq1}) is Ulam-Hyers stable on $\left[  0,T\right]  $
if there exists $C>0$ such that for any $\varepsilon>0$ and for any function
$z^{\ast}\left(  t\right)  $ satisfying inequality%
\begin{equation}
\left\Vert D_{0,t}^{\mu,\nu}z^{\ast}\left(  t\right)  +Az^{\ast}\left(
t\right)  +\Omega z^{\ast}\left(  t-h\right)  -f\left(  t\right)  \right\Vert
\leq\varepsilon\label{inq1}%
\end{equation}
and the initial conditions in (\ref{eq1}), there is a solution $z\left(
t\right)  $of (\ref{eq1}) such that
\[
\left\Vert z^{\ast}\left(  t\right)  -z\left(  t\right)  \right\Vert \leq
C\varepsilon
\]
for every $t\in\left[  0,T\right]  .$
\end{definition}

We reduce the notations of $Y_{\mu,\gamma}^{h}\left(  A,\Omega;t\right)  $,
$Q_{k,m}^{A,\Omega}$ to a mere $Y_{\mu,\gamma}^{h}\left(  t\right)  $,
$Q_{k,m}$ in the sequel.

\begin{theorem}
The following formulae hold:

\begin{enumerate}
\item The function $Y_{\mu,\gamma}^{h}\left(  \cdot\right)  $ is continuous on
$\left(  0,+\infty\right)  .$

\item $\frac{d}{dt}Y_{\mu,\gamma+1}^{h}\left(  t\right)  =Y_{\mu,\gamma}%
^{h}\left(  t\right)  ,\ \ \frac{d}{dt}Y_{\mu,\gamma+2}^{h}\left(  t\right)
=Y_{\mu,\gamma+1}^{h}\left(  t\right)  $ for all $t\in\mathbb{R}$.

\item $D_{0,t}^{\mu,\nu}Y_{\mu,\gamma}^{h}=-AY_{\mu,\gamma}^{h}\left(
t\right)  -\Omega Y_{\mu,\gamma}^{h}\left(  t-h\right)  .$
\end{enumerate}
\end{theorem}

\begin{proof}
The proofs of the properties $1.$ and $2.$ are obvious. Proof of the property
$3.$ is based on the following formula:%
\[
D_{0,t}^{\mu,\nu}t^{\alpha}=\frac{\Gamma\left(  \alpha+1\right)  }%
{\Gamma\left(  \alpha-\mu+1\right)  }t^{\alpha-\mu},\ \ t>0,\ n-1<\mu\leq
n,\ 0\leq\nu\leq1,\ \alpha>-1.
\]
\end{proof}

The main tool we use in this paper is the Laplace transform $F\left(
s\right)  :=L\left\{  f\left(  t\right)  \right\}  =\int_{0}^{\infty}%
e^{-st}f\left(  t\right)  dt,$ $\operatorname{Re}s>a$, which is defined for an
exponentially bounded function $f.$

\begin{lemma}
\label{lem:2}We have%
\begin{align*}
&  L^{-1}\left\{  \left(  e^{-hs}\left(  s^{\mu}I+A\right)  ^{-1}%
\Omega\right)  ^{m}s^{\mu-\gamma}\left(  s^{\mu}I+A\right)  ^{-1}\right\} \\
&  =%
{\displaystyle\sum\limits_{k=0}^{\infty}}
\left(  -1\right)  ^{k-m}Q_{k,m}\dfrac{\left(  t-mh\right)  _{+}^{k\mu
+\gamma-1}}{\Gamma\left(  k\mu+\gamma\right)  },
\end{align*}
where $Q_{k,m}$ is defined in (\ref{qq1}).
\end{lemma}

\begin{proof}
For $n=0$ by Lemma \ref{lem:1}(viii) we have%
\begin{align*}
L^{-1}\left\{  s^{\mu-\gamma}\left(  s^{\mu}I+A\right)  ^{-1}\right\}   &
=t^{\gamma-1}E_{\mu,\gamma}\left(  -At^{\mu}\right)  ,\ \ \ \\
L^{-1}\left\{  e^{-sh}\left(  s^{\mu}I+A\right)  ^{-\gamma}\right\}   &
=\left(  t-h\right)  _{+}^{\mu-1}E_{\mu,\mu}\left(  -A\left(  t-h\right)
^{\mu}\right)  ,\ \ \ t\geq h.
\end{align*}
Let $Q_{k,0}=A^{k}$. For $n=1$, we use the convolution property (Lemma
\ref{lem:1}(iii)) of the Laplace transform to get
\begin{align*}
&  L^{-1}\left\{  e^{-hs}\left(  s^{\mu}I+A\right)  ^{-1}\Omega s^{\mu-\gamma
}\left(  s^{\mu}I+A\right)  ^{-1}\right\} \\
&  =L^{-1}\left\{  e^{-hs}\left(  s^{\mu}I+A\right)  ^{-1}\Omega\right\}  \ast
L^{-1}\left\{  s^{\mu-\gamma}\left(  s^{\mu}I+A\right)  ^{-1}\right\} \\
&  =\int_{0}^{t}\left(  s-h\right)  _{+}^{\mu-1}E_{\mu,\mu}\left(  -A\left(
s-h\right)  ^{\mu}\right)  \Omega\left(  t-s\right)  ^{\gamma-1}E_{\mu,\gamma
}\left(  -A\left(  t-s\right)  ^{\mu}\right)  ds\\
&  =%
{\displaystyle\sum\limits_{k=0}^{\infty}}
{\displaystyle\sum\limits_{j=0}^{\infty}}
\frac{\left(  -1\right)  ^{k}A^{k}\Omega\left(  -1\right)  ^{j}A^{j}}%
{\Gamma\left(  \mu k+\mu\right)  \Gamma\left(  \mu j+\gamma\right)  }\int
_{h}^{t}\left(  s-h\right)  ^{\mu k+\mu-1}\left(  t-s\right)  ^{\mu
j+\gamma-1}ds\\
&  =%
{\displaystyle\sum\limits_{k=0}^{\infty}}
{\displaystyle\sum\limits_{j=0}^{\infty}}
\left(  -1\right)  ^{k}A^{k}\Omega\left(  -1\right)  ^{j}A^{j}\frac{\left(
t-h\right)  _{+}^{\mu k+\mu j+\mu+\gamma-1}}{\Gamma\left(  \mu k+\mu
j+\mu+\gamma\right)  }\\
&  =%
{\displaystyle\sum\limits_{k=0}^{\infty}}
\left(  -1\right)  ^{k}%
{\displaystyle\sum\limits_{j=0}^{k}}
A^{k-j}\Omega A^{j}\frac{\left(  t-h\right)  _{+}^{\mu k+\mu+\gamma-1}}%
{\Gamma\left(  \mu k+\mu+\gamma\right)  }\\
&  =%
{\displaystyle\sum\limits_{k=1}^{\infty}}
\left(  -1\right)  ^{k-1}%
{\displaystyle\sum\limits_{j=0}^{k-1}}
A^{k-1-j}\Omega A^{j}\frac{\left(  t-h\right)  _{+}^{\mu k+\gamma-1}}%
{\Gamma\left(  \mu k+\gamma\right)  }.
\end{align*}
Now, to use the mathematical induction, suppose that it holds for $n=m$. Then
convolution property yields%
\begin{gather*}
L^{-1}\left\{  \left(  e^{-hs}\left(  s^{\mu}I+A\right)  ^{-1}\Omega\right)
^{m+1}s^{\mu-\gamma}\left(  s^{\mu}I+A\right)  ^{-1}\right\} \\
=L^{-1}\left\{  e^{-hs}\left(  s^{\mu}I+A\right)  ^{-1}\Omega\right\}  \ast
L^{-1}\left\{  \left(  e^{-hs}s^{-\beta}\left(  s^{\mu}I+A\right)  ^{-1}%
\Omega\right)  ^{m}s^{\mu-\gamma}\left(  s^{\mu}I+A\right)  ^{-1}\right\} \\
=\int_{h}^{t}\left(  s-h\right)  _{+}^{\mu-1}E_{\mu,\mu}\left(  -A\left(
s-h\right)  ^{\mu}\right)  \Omega%
{\displaystyle\sum\limits_{j=0}^{\infty}}
\left(  -1\right)  ^{j}Q_{j+m,m}\frac{\left(  t-s-mh\right)  _{+}^{\mu j+\mu
m+\gamma-1}}{\Gamma\left(  \mu j+\mu m+\mu\right)  }ds\\
=%
{\displaystyle\sum\limits_{k=0}^{\infty}}
{\displaystyle\sum\limits_{j=0}^{\infty}}
\left(  -1\right)  ^{k}A^{k}\Omega\left(  -1\right)  ^{j}Q_{j+m,m}%
\int\limits_{h}^{t}\dfrac{\left(  t-s-h\right)  _{+}^{k\mu+\mu-1}}%
{\Gamma\left(  k\mu+\mu\right)  }\frac{\left(  s-mh\right)  _{+}^{\mu j+\mu
m+\gamma-1}}{\Gamma\left(  \mu j+\mu m+\gamma\right)  }ds\\
=%
{\displaystyle\sum\limits_{k=0}^{\infty}}
{\displaystyle\sum\limits_{j=0}^{\infty}}
\left(  -1\right)  ^{k}A^{k}\Omega\left(  -1\right)  ^{j}Q_{j+m,m}%
\int\limits_{mh}^{t-h}\dfrac{\left(  t-s-h\right)  ^{k\mu+\mu-1}}%
{\Gamma\left(  k\mu+\mu\right)  }\frac{\left(  s-mh\right)  ^{\mu j+\mu
m+\gamma-1}}{\Gamma\left(  \mu j+\mu m+\gamma\right)  }ds\\
=%
{\displaystyle\sum\limits_{k=0}^{\infty}}
{\displaystyle\sum\limits_{j=0}^{\infty}}
\left(  -1\right)  ^{k}A^{k}\Omega\left(  -1\right)  ^{j}Q_{j+m,m}%
\dfrac{\left(  t-\left(  m+1\right)  h\right)  _{+}^{k\mu+j\mu+\left(
m+1\right)  \mu+\gamma-1}}{\Gamma\left(  k\mu+j\mu+\left(  m+1\right)
\mu+\gamma\right)  }\\
=%
{\displaystyle\sum\limits_{k=m+1}^{\infty}}
\left(  -1\right)  ^{k-m-1}%
{\displaystyle\sum\limits_{j=0}^{k-m-1}}
A^{k-j}\Omega Q_{j+m,m}\dfrac{\left(  t-\left(  m+1\right)  h\right)
_{+}^{k\mu+j\mu+\left(  m+1\right)  \mu+\gamma-1}}{\Gamma\left(  k\mu
+j\mu+\left(  m+1\right)  \mu+\gamma\right)  }%
\end{gather*}
what was to be proved.
\end{proof}

\begin{lemma}
\label{lem:4}We have%
\begin{align*}
Y_{\mu,\gamma}^{h}\left(  t\right)   &  :=L^{-1}\left\{  s^{\mu-\gamma}\left(
s^{\mu}I+A+\Omega e^{-hs}\right)  ^{-1}\right\} \\
&  =%
{\displaystyle\sum_{m=0}^{\infty}}
{\displaystyle\sum\limits_{k=0}^{\infty}}
\left(  -1\right)  ^{k}Q_{k,m}\dfrac{\left(  t-mh\right)  _{+}^{k\mu+\gamma
-1}}{\Gamma\left(  k\mu+\gamma\right)  }.
\end{align*}

\end{lemma}

\begin{proof}
It is easy to see that%
\begin{align*}
&  L^{-1}\left\{  s^{\mu-\gamma}\left(  s^{\mu}I+A+\Omega e^{-hs}\right)
^{-1}\right\} \\
&  =L^{-1}\left\{  s^{\mu-\gamma}\left(  \left(  s^{\mu}I+A\right)  I+\left(
s^{\mu}I+A\right)  \left(  s^{\mu}I+A\right)  ^{-1}\Omega e^{-hs}\right)
^{-1}\right\} \\
&  =L^{-1}\left\{  \left(  I+\left(  s^{\mu}I+A\right)  ^{-1}\Omega
e^{-hs}\right)  ^{-1}s^{\mu-\gamma}\left(  s^{\mu}I+A\right)  ^{-1}\right\} \\
&  =L^{-1}\left\{
{\displaystyle\sum_{m=0}^{\infty}}
e^{-mhs}\left(  -1\right)  ^{n}\left(  \left(  s^{\mu}I+A\right)  ^{-1}%
\Omega\right)  ^{m}s^{\mu-\gamma}\left(  s^{\mu}I+A\right)  ^{-1}\right\} \\
&  =%
{\displaystyle\sum_{m=0}^{\infty}}
L^{-1}\left\{  e^{-mhs}\left(  -1\right)  ^{m}\left(  \left(  s^{\mu
}I+A\right)  ^{-1}\Omega\right)  ^{m}s^{\mu-\gamma}\left(  s^{\mu}I+A\right)
^{-1}\right\}  .
\end{align*}
Hence, by Lemma \ref{lem:2} we have%
\[
L^{-1}\left\{  s^{\mu-\gamma}\left(  s^{\mu}I+A+\Omega e^{-hs}\right)
^{-1}\right\}  =%
{\displaystyle\sum_{m=0}^{\infty}}
{\displaystyle\sum\limits_{k=0}^{\infty}}
\left(  -1\right)  ^{k}Q_{k,m}\dfrac{\left(  t-mh\right)  _{+}^{k\mu+\gamma
-1}}{\Gamma\left(  k\mu+\gamma\right)  }.
\]
\end{proof}

\section{Exact analytical solution and Ulam-Hyers stability}

We obtain the exact analytical solution of the Hilfer type fractional second
order problem (\ref{eq1}) using delayed Mittag-Leffler type matrix function
$Y_{\mu,\gamma}^{h}\left(  A,\Omega;t\right)  $ and study their Ulam-Hyers stability.

\begin{theorem}
\label{thm:1}The analytical solution of the IVP problem (\ref{eq1}) has the
form%
\begin{align*}
z\left(  t\right)   &  =Y_{\mu,\left(  \mu-2\right)  \left(  1-\nu\right)
+1}^{h}\left(  t\right)  \left(  I_{0,t}^{\left(  1-\nu\right)  \left(
2-\mu\right)  }\varphi\right)  \left(  0\right) \\
&  +Y_{\mu,\left(  \mu-2\right)  \left(  1-\nu\right)  +2}^{h}\left(
t\right)  \left(  I_{0,t}^{\left(  1-\nu\right)  \left(  2-\mu\right)
-1}\varphi\right)  \left(  0\right) \\
&  -\int_{-h}^{0}Y_{\mu,\mu}^{h}\left(  t-s-h\right)  \Omega\varphi\left(
s\right)  ds+\int_{0}^{t}Y_{\mu,\mu}^{h}\left(  t-s\right)  f\left(  s\right)
ds.
\end{align*}

\end{theorem}

\begin{proof}
Assume that the function $f$ and the solution of (\ref{eq1}) is exponentially
bounded. By applying the Laplace transform to the both sides of (\ref{eq1}),
we obtain the following relation%
\[
L\left\{  D_{0,t}^{\mu,\nu}z\left(  t\right)  \right\}  +AL\left\{  z\left(
t\right)  \right\}  +\Omega L\left\{  z\left(  t-h\right)  \right\}
=L\left\{  f\left(  t\right)  \right\}  .
\]
It follows that%
\begin{align*}
\left(  s^{\mu}I+A+\Omega e^{-hs}\right)  Z\left(  s\right)   &  =s^{2\left(
1-\nu\right)  +\mu\nu-1}\left(  I_{0,t}^{\left(  1-\nu\right)  \left(
2-\mu\right)  }\varphi\right)  \left(  0\right) \\
&  +s^{2\left(  1-\nu\right)  +\mu\nu-2}\left(  I_{0,t}^{\left(  1-\nu\right)
\left(  2-\mu\right)  -1}\varphi\right)  \left(  0\right) \\
&  -\Omega\int_{0}^{\infty}e^{-st}z\left(  t-h\right)  dt+F\left(  s\right)  ,
\end{align*}
where $Z\left(  s\right)  =L\left\{  z\left(  t\right)  \right\}  ,\ F\left(
s\right)  =L\left\{  f\left(  t\right)  \right\}  $. For sufficiently large
$s$, such that%
\[
\left\Vert A+\Omega e^{-hs}\right\Vert <s^{\mu},
\]
the matrix $s^{\mu}I+A+\Omega e^{-hs}$ is invertible and it holds that
\begin{align*}
Z\left(  s\right)   &  =s^{2\left(  1-\nu\right)  +\mu\nu-1}\left(  s^{\mu
}I+A+\Omega e^{-hs}\right)  ^{-1}\left(  I_{0,t}^{\left(  1-\nu\right)
\left(  m-\mu\right)  }\varphi\right)  \left(  0\right) \\
&  +s^{2\left(  1-\nu\right)  +\mu\nu-2}\left(  s^{\mu}I+A+\Omega
e^{-hs}\right)  ^{-1}\left(  I_{0,t}^{\left(  1-\nu\right)  \left(
m-\mu\right)  -1}\varphi\right)  \left(  0\right) \\
&  -\left(  s^{\mu}I+A+\Omega e^{-hs}\right)  ^{-1}\Omega\Psi\left(  s\right)
\\
&  +\left(  s^{\mu}I+A+\Omega e^{-hs}\right)  ^{-1}F\left(  s\right)  .
\end{align*}
By Lemma \ref{lem:4}%
\begin{align}
z\left(  t\right)   &  =Y_{\mu,\left(  \mu-2\right)  \left(  1-\nu\right)
+1}^{h}\left(  t\right)  \left(  I_{0,t}^{\left(  1-\nu\right)  \left(
m-\mu\right)  }\varphi\right)  \left(  0\right) \nonumber\\
&  +Y_{\mu,\left(  \mu-2\right)  \left(  1-\nu\right)  +2}^{h}\left(
t\right)  \left(  I_{0,t}^{\left(  1-\nu\right)  \left(  m-\mu\right)
-1}\varphi\right)  \left(  0\right) \nonumber\\
&  -\int_{-h}^{0}Y_{\mu,\mu}^{h}\left(  t-s-h\right)  \Omega\varphi\left(
s\right)  ds+\int_{0}^{t}Y_{\mu,\mu}^{h}\left(  t-s\right)  f\left(  s\right)
ds, \label{qq2}%
\end{align}
since%
\begin{align*}
&  L^{-1}\left\{  \left(  s^{\mu}I+A+\Omega e^{-hs}\right)  ^{-1}\Omega
\Psi\left(  s\right)  \right\}  =L^{-1}\left\{  \left(  s^{\mu}I+A+\Omega
e^{-hs}\right)  ^{-1}\right\}  \ast L^{-1}\left\{  \Omega\Psi\left(  s\right)
\right\} \\
&  =\int_{0}^{t}Y_{\mu,\mu}^{h}\left(  t-s\right)  \Omega\psi\left(
s-h\right)  ds=\int_{0}^{h}Y_{\mu,\mu}^{h}\left(  t-s\right)  \Omega
\varphi\left(  s-h\right)  ds\\
&  =\int_{-h}^{0}Y_{\mu,\mu}^{h}\left(  t-s-h\right)  \Omega\varphi\left(
s\right)  ds.
\end{align*}
Now the assumption on the exponential boundedness can be omitted. We can
easily check that (\ref{qq2}) is a solution of (\ref{eq1}).
\end{proof}

\begin{theorem}
Let $1<\mu<2$, $0\leq\nu\leq1,$ $f\in C\left(  \left[  0,\infty\right)
,\mathbb{R}^{d}\right)  .$ System (\ref{eq1}) is stable in Ulam-Hyers sence on
$[0,T]$.
\end{theorem}

\begin{proof}
Let $z^{\ast}\left(  t\right)  $ satisfy the inequality (\ref{inq1}) and the
initial conditions in (\ref{eq1}). Set%
\[
X\left(  t\right)  =D_{0,t}^{\mu,\nu}z^{\ast}\left(  t\right)  +Az^{\ast
}\left(  t\right)  +\Omega z^{\ast}\left(  t-h\right)  -f\left(  t\right)
,\ t\in\left[  0,T\right]  .
\]
It follows from definition \ref{def:2} that $\left\Vert X\left(  t\right)
\right\Vert <\varepsilon$. By Theorem \ref{thm:1} we have
\begin{align*}
z^{\ast}\left(  t\right)   &  =Y_{\mu,\left(  \mu-2\right)  \left(
1-\nu\right)  +1}^{h}\left(  t\right)  \left(  I_{0,t}^{\left(  1-\nu\right)
\left(  m-\mu\right)  }\varphi\right)  \left(  0\right) \\
&  +Y_{\mu,\left(  \mu-2\right)  \left(  1-\nu\right)  +2}^{h}\left(
t\right)  \left(  I_{0,t}^{\left(  1-\nu\right)  \left(  m-\mu\right)
-1}\varphi\right)  \left(  0\right) \\
&  -\int_{-h}^{0}Y_{\mu,\mu}^{h}\left(  t-s-h\right)  \Omega\varphi\left(
s\right)  ds+\int_{0}^{t}Y_{\mu,\mu}^{h}\left(  t-s\right)  \left(  f\left(
s\right)  -X\left(  s\right)  \right)  ds.
\end{align*}
Thus we can estimate the difference $z^{\ast}\left(  t\right)  -z\left(
t\right)  $ as follows%
\[
\left\Vert z^{\ast}\left(  t\right)  -z\left(  t\right)  \right\Vert
=\left\Vert \int_{0}^{t}Y_{\mu,\mu}^{h}\left(  t-s\right)  X\left(  s\right)
ds\right\Vert \leq\varepsilon\int_{0}^{T}\left\Vert Y_{\mu,\mu}^{h}\left(
T-s\right)  \right\Vert ds=C\varepsilon.
\]
Then the problem (\ref{eq1}) is Ulam-Hyers stable on $[0,T]$.
\end{proof}

\section{Conclusion}

The article solves a problem of finding exact analytical solution of
continuous linear time-delay systems using the delayed Mittag-Leffler type
matrix functions of two variables. In articles \cite{12}, \cite{4} delayed
exponential is suggested to obtain an exact solution of delayed first order
continuous equations. Similar results for sequential Caputo type and
Riemann-Liouville type fractional linear time-delay systems of order
$1<2\alpha<2$ were obtained in \cite{16}, \cite{17}. These results are
obtained either for systems with pure delay or under the condition of
commutativity of $A$ and $\Omega$. In this article we drop the commutativity
condition. The result has been obtained by defining the new delayed
Mittag-Leffler matrix function and employing the Laplace transform. The work
contained in this article will be useful for future research on fractional
time-delay systems.

\end{document}